\documentclass{article}

\usepackage{arxiv}

\usepackage[utf8]{inputenc} 
\usepackage[T1]{fontenc}    
\usepackage{hyperref}       
\usepackage{url}            
\usepackage{booktabs}       
\usepackage{amsfonts}       
\usepackage{nicefrac}       
\usepackage{microtype}      
\usepackage{lipsum}		

\usepackage{graphicx}
\usepackage{natbib}
\usepackage{doi}
\usepackage{xcolor,colortbl}

\usepackage[linesnumbered,ruled,vlined]{algorithm2e}
\usepackage{amsthm}

\newtheorem{theorem}{Theorem}[section]
\newtheorem{lemma}[theorem]{Lemma}

\SetCommentSty{mycommfont}

\SetKwInput{KwInput}{Input}                
\SetKwInput{KwOutput}{Output} 
\SetKwInput{Kwset}{Set}

\usepackage{xcolor,colortbl}

\newcolumntype{a}{>{\columncolor{yellow}}c}
\newcolumntype{b}{>{\columncolor{green}}c}

\usepackage{pgf-pie}
\usepackage{array,longtable}
\usepackage{xcolor,colortbl}
\usepackage{amsthm,amssymb,amsmath,mathtools}

\title{Barzilai and Borwein  conjugate gradient method equipped with a non-monotone line search technique and its application on non-negative matrix factorization}

\author{ \href{https://orcid.org/0000-0002-5731-8234}{\includegraphics[scale=0.06]{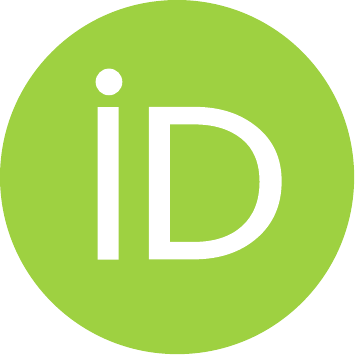}\hspace{1mm} Sajad Fathi Hafshejani}\\
    Department of Math and Computer Science\\
	University of Lethbridge\\
	 Lethbridge, AB, Canada\\
	\texttt{sajad.fathihafshejan@uleth.ca} \\
	\And
	{Daya Gaur} \\
Department of Math and Computer Science\\
	University of Lethbridge\\
	 Lethbridge, AB, Canada\\
	\texttt{daya.gaur@uleth.ca} \\
	
		\And
	{Shahadat Hossain} \\
    Department of Math and Computer Science\\
	University of Lethbridge\\
	 Lethbridge, AB, Canada\\
	\texttt{shahadat.hossain@uleth.ca} \\
	
		\And
	{Robert Benkoczi} \\
    Department of Math and Computer Science\\
	University of Lethbridge\\
	 Lethbridge, AB, Canada\\
	\texttt{robert.benkoczi@uleth.ca } \\
}

\renewcommand{\shorttitle}{\textit{arXiv} Template}

\hypersetup{
pdftitle={Barzilai and Borwein  conjugate gradient method equipped with a non-monotone line search technique and its application on non-negative matrix factorization},
pdfsubject={},
pdfauthor={S. Fathi Hafshejani,D. Gaur,S. Hossain, R. Benkoczi},
pdfkeywords={Barzilai and Borwein method, Conjugate gradient decent, Non-monotone line search,Non-negative matrix factorization},
}

\begin{document}
\maketitle

\begin{abstract}
In this paper, we propose a new non-monotone conjugate gradient method for solving unconstrained nonlinear optimization problems. We first modify the non-monotone line search method by introducing a new trigonometric function to calculate the non-monotone parameter, which plays an essential role in the algorithm's efficiency. Then, we apply a convex combination of the Barzilai-Borwein method \citep{Barzilai} for calculating the value of step size in each iteration. Under some suitable assumptions,  we prove that the new algorithm has the global convergence property. The efficiency and effectiveness of the proposed method are determined in practice by   applying the algorithm to some standard test problems and  non-negative matrix factorization problems.
\end{abstract}

\section{Introduction}
In this paper, we are interested to solve the following unconstrained optimization problem:
\begin{eqnarray}\label{general}
\min_{x\in\Bbb{R}^n}f(x),
\end{eqnarray}
in which $f:\Bbb{R}^n\rightarrow \Bbb{R}$ is a continuously differentiable function. There are various iterative 
approaches for solving (\ref{general}) \citep{Nocedal}. The Conjugate Gradient (CG) method is one such approach. The CG based methods
do not need any second-order information of the objective function. For a given point $x_0\in \Bbb{R}^n$, the iterative formula
describing the  CG method is:
\begin{equation}\label{iter}
x_{k+1}=x_k+\alpha_k d_k,
\end{equation}
in which $x_k$ is current iterate point, $\alpha_k$ is the step size, and $d_k$ is the search direction determined by:
\begin{eqnarray}\label{dk}
d_k=\left\{
\begin{array}{lr}
-g_k\qquad\qquad\quad\qquad k=0,&\\
-g_k+\beta_{k-1}d_{k-1}\qquad k\geq 1,&\\
\end{array} \right.
\end{eqnarray}
where $g_k=\nabla f(x_k)$ is the gradient of the objective function in the current iteration. The conjugate gradient parameter is $\beta_k$, whose choice of different values leads to various CG methods.
The most well-known of the CG methods are the Hestenes-Stiefel (HS) method \citep{hestenes}, Fletcher-Reeves (FR) method \citep{Fletcher64},
Conjugate Descent (CD) \citep{Fletcher13}, and Polak-Ribiere-Polyak (PRP) \citep{prp}.

There are various approaches to determining a suitable  step size in each iteration such as Armijo line search, Goldstein line search, and Wolfe line search \citep{Nocedal}. The Armijo line search finds the largest value of step size in each iteration such that the following inequality holds:
\begin{eqnarray}\label{line}
f(x_k+\alpha_kd_k)\leq f(x_k)+\gamma\alpha_kg_k^Td_k
\end{eqnarray}
in which $\gamma\in(0,1)$ is a constant parameter.
Grippo et al. \citep{Grippo86} introduced a non-monotone Armijo-type line search technique as another way to compute step size.
The Incorporation of  the non-monotone strategy into the gradient and projected gradient  approaches, the conjugate gradient method, and the trust-region methods  has led to significant improvements to these methods. Zhang and Hager \citep{Zhang04} gave some conditions to improve the convergence rate of this strategy. Ahookhosh et al. \citep{Ahookhosh122} built on these results and investigated   a new non-monotone condition:
\begin{equation}\label{amin}
f(x_k+\alpha_kd_k)\leq R_k+\gamma\alpha_kg_k^T d_k,
\end{equation}
where $R_k$ is defined by
\begin{eqnarray}
& R_k=\eta_k f_{l_k}+(1-\eta_k)f_k, & \label{rk}\label{flk}\\
& \eta_k\in[\eta_{\min},\eta_{\max}],~\eta_{\min}\in[0, 1), \ \eta_{\max}\in[\eta_{\min},1], &  \notag\\
& f_{l_k}=\max_{0\leq j\leq m_k}\{f_{k-j}\}, \nonumber \\
&  m_0=0, \ \ 0\leq m_k\leq \min\{m_{k-1}+1,N\} \mbox{ for some } N\geq 0. 
\end{eqnarray}
Note that $\eta_k$ is known as the non-monotone parameter and plays an essential role in the algorithm’s convergence.

Although this new non-monotone strategy in \citep{Ahookhosh122} has some appealing properties, especially in functional performance, current algorithms based on this non-monotone strategy
face the following challenges.
\begin{itemize}
    \item The existing schemes for determining the parameter $\eta_k$ 
may not reduce the value of the objective function significantly in initial iterations.
To overcome this drawback,  we propose a new scheme for choosing $\eta_k$ 
based on the gradient behaviour  of the objective function.
This can reduce the total  number of iterations.
\item Many evaluations of the objective function are needed to find 
the step length $\alpha_k$ in step $k$. 
To make this step more efficient,  we use an adaptive and composite step length procedure from \citep{Li19} to determine   the initial value of the step length in inner iterations.
\item The third issue is the global convergence for the non-monotone CG method. Most exiting CG methods use the Wolfe condition, which plays a vital role in establishing the global convergence of various CG methods \citep{Nazareth01}. Wolfe line search is more expensive than the Armijo line search strategy. Here, we define a suitable conjugate gradient parameter so that the scheme proposed here has global convergence property.

\end{itemize}

By combining the outlined strategies, we propose a modification to the non-monotone line search method. Then, we incorporate this approach into the CG method and introduce a new non-monotone CG algorithm. We prove that our proposed algorithm has  global convergence. Finally, we compare our algorithm and eight other algorithms on standard tests  and non-negative matrix factorization instances. We utilize some criteria such as the number of objective function evaluations, the number of gradient evaluations, the number of iterations, and the CPU time to compare the performance of algorithms.

\section{An improved non-monotone line search algorithm} \label{s:algorithm}

This section discusses the issues with the state of the art of non-monotone line search strategy, choice of the step sizes, and finally, the conjugate gradient parameter.
\subsection{A new scheme of choosing $\eta_k$}
Recall that the   non-monotone line search strategy is determined by equation \eqref{amin} in step $k$.
The parameter $\eta_k$ is involved in the non-monotone term (\ref{flk})
and its choice can have a significant impact on the performance of the algorithm. There are two common approaches for calculating 
$\eta_k$.
The scheme proposed by Ahookhosh et al. \citep{Ahookhosh122}  has been used in most of the existing non-monotone algorithms \citep{Esmaeili,Ahookhosh_Nu,Amini_App14,Ahookhosh15}.
This strategy can be formulated as $\eta_k=\frac{1}{3}\eta_0 (-\frac{1}{2})^k+\frac{2}{3}\eta_0$ 
where $\eta_0=0.15$ and the limit value of $\eta_k$ is 0.1.
The other scheme proposed by Amini et al. \citep{Amini14}, which depends on the behaviour of gradient is given by:
\begin{equation}\label{Amini's_method}
\eta_0=0.95, \ \
\eta_{k}= \left\{
\begin{array}{ll}
\frac{2}{3}\eta_{k-1} +0.01, & \mbox{if }  ~\|g_{k} \|_{\infty}\leq 10^{-3}; \\
\max\{ 0.99\eta_{k-1},0.5\}, & \mbox{otherwise}.
\end{array} \right.
\end{equation}
To illustrate the behaviour of $\eta_k$ proposed in   \citep{Ahookhosh122} and \citep{Amini14}, we solve the problem $f(x)= (x_0-5)^2+\sum_{i=1}^{40} (x_i-1)^2$ for $ x\in \Bbb{R}^{41}$. 
The values of the parameter $\eta_k$ corresponding to the two schemes are displayed in Fig.  \ref{muk} (Left).
 \begin{figure}[h!]
\centering
 \includegraphics[width=.45\textwidth]{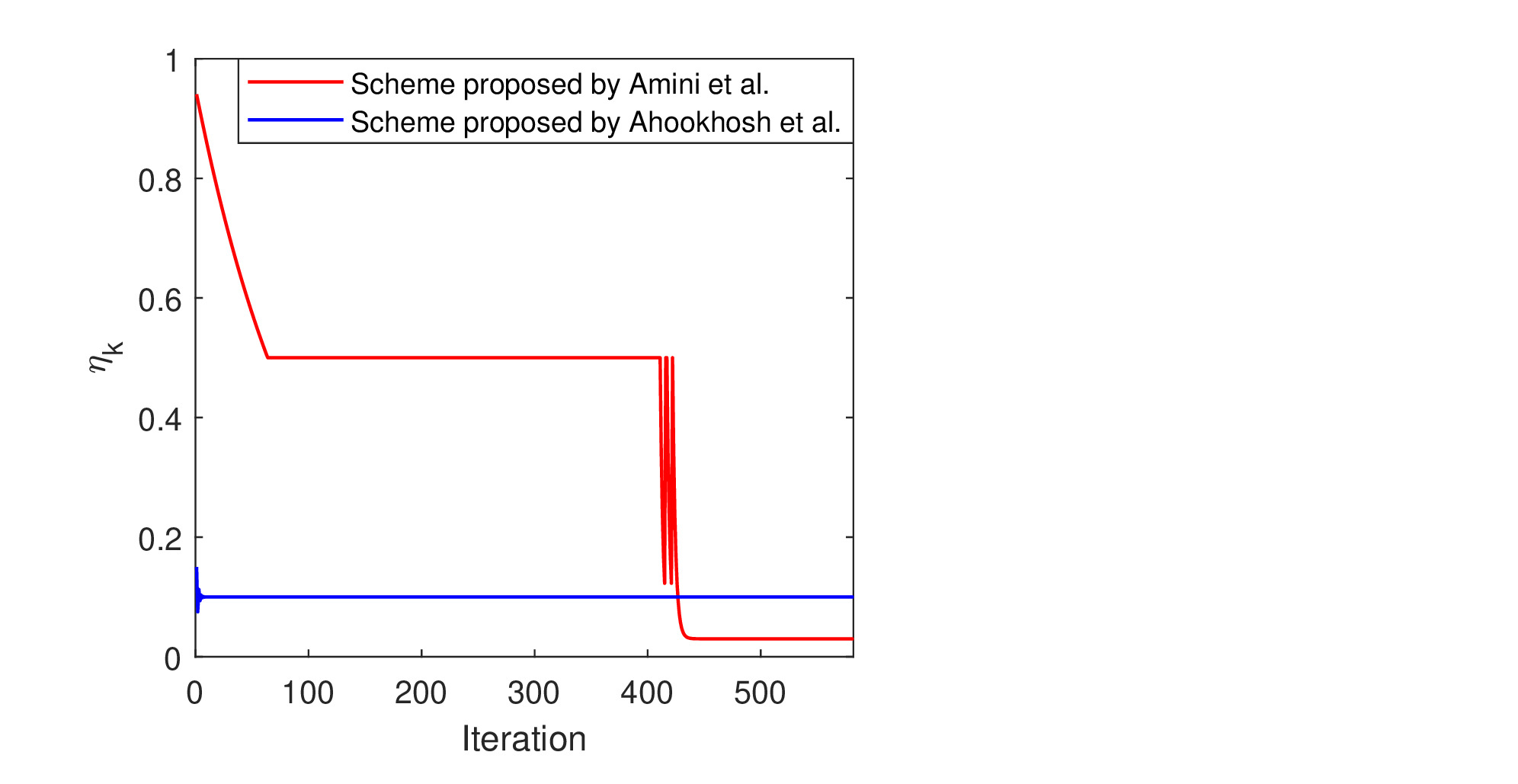}
   \includegraphics[width=.45\textwidth]{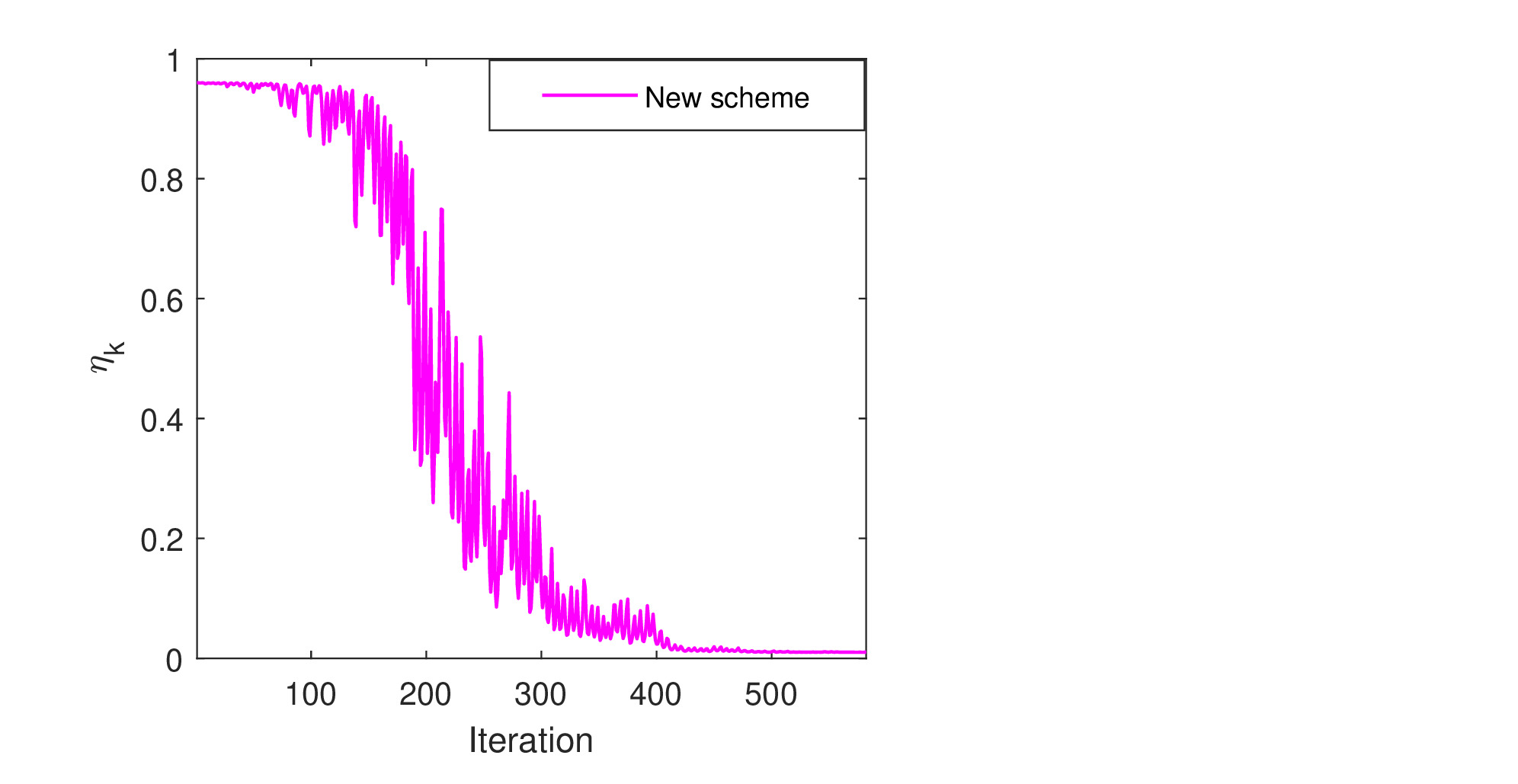}
\caption{(Left): Values of $\eta_k$ proposed in   \citep{Ahookhosh122} and \citep{Amini14}, (Right): Values of $\eta_k$ for the new scheme.}
\label{muk}
\end{figure}
As shown in Fig. \ref{muk},  for the scheme proposed by Ahookhosh et al. \citep{Ahookhosh122}, $\eta_k$
is close to $0.1$ after only a few  iterations. {Notice that  $\eta_k$ in each iteration does not have any connection with the behaviour of the objective function. Thus this scheme is not effective.} In addition, there are two issues with the  scheme introduced by Amini et al. in \citep{Amini14}.
One problem indicated by Fig. \ref{muk} is that that $\eta_k$ decreases relatively
quickly for the first 65 iterations.
{Since the algorithm requires the long iterations to solve his problem}, ideally $\eta_k$ should be close to 1 for these initial iterations.
The second problem is that  the value of $\eta_k$ remains the same  for a large number of iterations and it is not affected by  the behaviour of the objective function.

To avoid theses challenges, we propose an adaptive strategy for calculating the value of  $\eta_k$:
\begin{equation} \label{eq:etakn}
\eta_{k}=0.95\sin\left(\frac{\pi \|g_{k}\|}{1+2\|g_{k}\|}\right)+0.01.
\end{equation}
When $x_k$ is far away from the minimizer, we can reasonably assume that  $\|g_k\|$ is large. Thus the value  of $\eta_k$ defined by \eqref{eq:etakn} is close to 1.
This makes the scheme closer to the original non-monotone strategy in the initial iterations, providing a chance to reduce the value of the objective function more significantly in the initial iterations. On the other hand, when $x_k$ is close to the minimizer, $\|g_k\|$ is small, then the value of $\eta_k$ is close to zero. Thus, the step length is small so that the new point stays in the neighbourhood of the optimal point. Thus the new scheme is closer to the  monotone strategy. We plot the behaviour of  $\eta_k$ denoted by \eqref{eq:etakn} in Fig. \ref{muk} (Right), using the same values of the gradient for the optimization problem mentioned above.

\subsection{ New schemes for choosing $\alpha_k$ }
We utilize a convex combination of the  Barzilai-Borwein (BB) step sizes to calculate an appropriate $\alpha_k$ in each outer iteration as in  \citep{Li19}. Our strategy calculates the value of  $\alpha_k$,   using the following equation:
\begin{equation}\label{newalpha}
\alpha_k^{{\scriptscriptstyle \textrm{CBB}}} =\mu_k\alpha^{(1)}_k+(1-\mu_k)\alpha^{(2)}_k,
\end{equation}
where
\begin{eqnarray*}
	&\alpha_k^{(1)}=\frac{s_k^Ts_k}{s_k^Ty_k},\quad \alpha^{(2)}_k=\frac{s_k^Ty_k}{y_k^Ty_k},\quad s_k:=x_k-x_{k-1},\quad y_k:=g_k-g_{k-1};&\\
	&\mu_k=\frac{K_2}{K_1+K_2}\quad
	K_1=\|\alpha^{(1)}_k y_k-s_k\|^2,\quad K_2=\|(\alpha^{(2)}_k)^{-1}s_k-y_k\|^2.&
\end{eqnarray*}
\subsection{Conjugate gradient parameter}
Here, we propose the new conjugate gradient parameter given by:
\begin{eqnarray}\label{cgpar}
\beta_k=\omega \frac{\|g_k\|}{\|d_{k-1}\|},\quad \omega \in (0,1).
\end{eqnarray}
The complete algorithm  is in Appendix \ref{AppA} (see Algorithm \ref{alg1}). The next lemma proves a key property of   $\beta_k$ which is very important in proving the algorithm’s convergence. The proofs are in the Appendix \ref{AppA}.
\begin{lemma}\label{decent}
For the  search direction $d_k$ and the constant $c>0$ we have:
	\begin{eqnarray}
	d_k^Tg_k\leq -c\|g_k\|.
	\end{eqnarray}
\end{lemma}
 The following assumptions are used to analyze the convergence properties of  Algorithm \ref{alg1}.
\begin{description}
	\item[H1] The level set $
	\mathcal{L}(x_0)=\{x|f(x)\leq f(x_0),~~~~x\in \Bbb{R}^n\}$ is bounded set.
	\item[H2] The gradient of objective function is Lipschitz continuous over an open convex set $C$ containing $	\mathcal{L}(x_0)$. That is:
	\begin{equation*}
	\|g(x)-g(y)\|\leq L\|x-y\|,\qquad \forall ~x,y\in C.
	\end{equation*}
\end{description}
We prove the following Theorem about the global convergence of Algorithm \ref{alg1}, the proof of which follows from the Lemmas presented in 
this section. Please see the appendix for the proofs.

\begin{theorem}\label{glob}
	{Let $(H1)$, $(H2)$, and Lemmas \ref{decent} and \ref{aboveserch} hold. Then, for the
	sequence $\{x_k\}$ generated by Algorithm \ref{alg1}, we have $\lim_{k\rightarrow \infty} \|g_k\|=0.$
}\end{theorem}

\begin{lemma}\label{aboveserch}
	Suppose that the search direction $d_k$ with the CG parameter $\beta_k$ given by (\ref{cgpar}) is generated by  Algorithm \ref{alg1}. Then, an upper bound for $d_k$ is given by $\|d_k\|\leq (1+\omega)\|g_k\|.$
\end{lemma}
\begin{lemma}\label{low-bou}
	Suppose that $x_k$ is not a stationary point of (\ref{general}). Then there exists a constant
	\begin{equation*}
	{\lambda}=\min \left\{\beta_1\rho,\frac{2(1-\omega)\rho(1-\gamma)}{L(1+\omega)^2}\right\},
	\end{equation*}
	such that $\alpha_k\geq {\lambda}$.
\end{lemma}

\section{Numerical Results}
In this section we test the new algorithm to solve a set of standard optimization problems and the non-negative matrix factorization problem, which is a non-convex optimization problem. The implementation level details are  in Appendix \ref{AppB}.
To  demonstrate the efficiency of the   proposed algorithm, we compare  our algorithm and eight other existing algorithms introduced in \citep{Ahookhosh122,Amini14,Jiang,Zhang} on a set of $110$ standards test problems. To describe the behaviour of each strategy, we use
performance profiles proposed by  Dolan and Moré \citep{Dolan}.
Note that the performance profile for an algorithm  $p_s(\tau): \Bbb{R}\mapsto [0, 1]$  is a non-decreasing, piece-wise constant function, continuous from the right at each breakpoint. Moreover, the value  $p_s(1)$ denotes the probability that the algorithm will win against the rest of the algorithm. More information  on the  performance profile is in Appendix \ref{AppB}. We plot  the performance profile of each algorithm in terms of the total number of outer iteration and the  CPU time on the set of standard test problems in Fig. \ref{results}.  
 \begin{figure}[h!]
\centering
 \includegraphics[width=.45\textwidth]{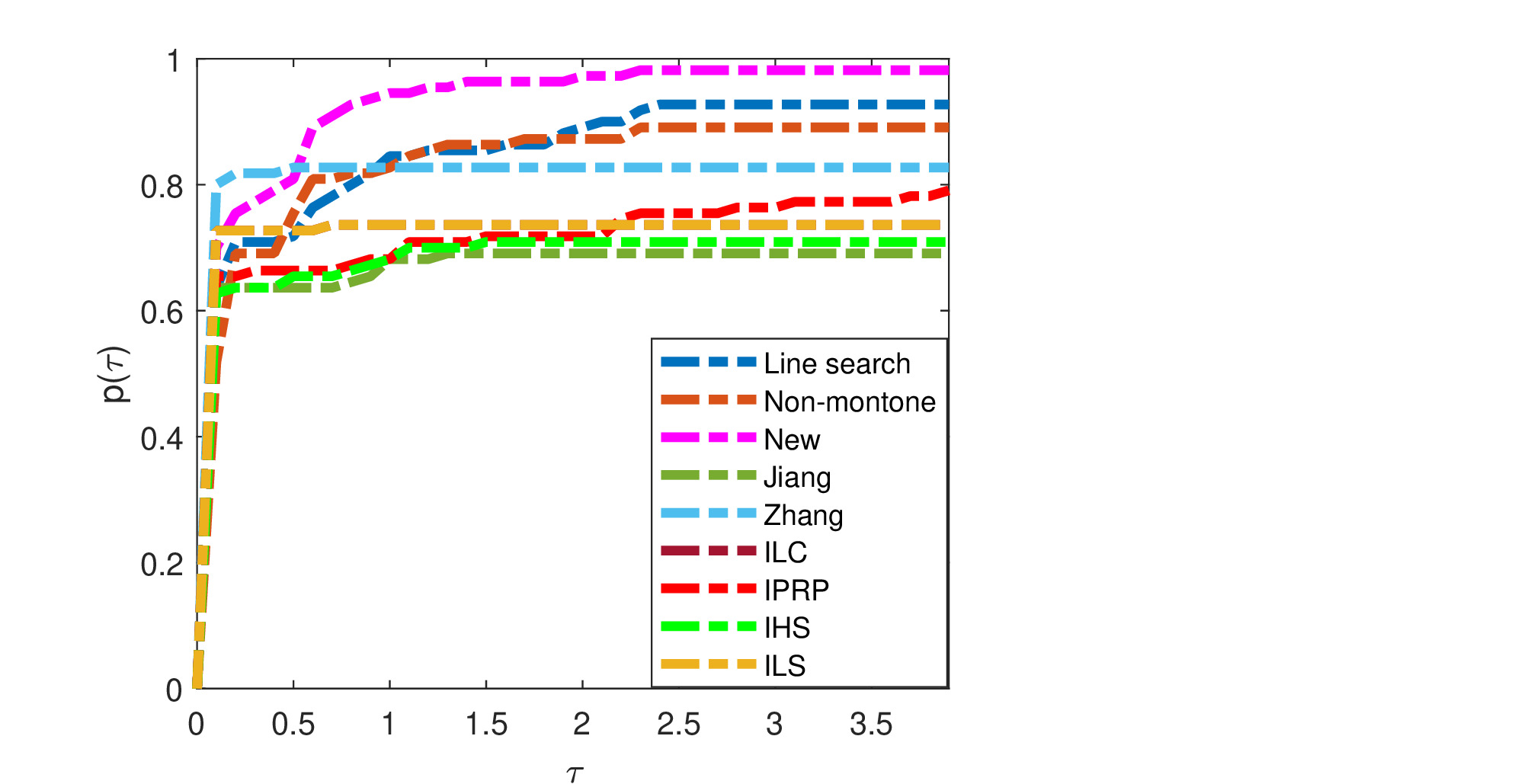}
   \includegraphics[width=.45\textwidth]{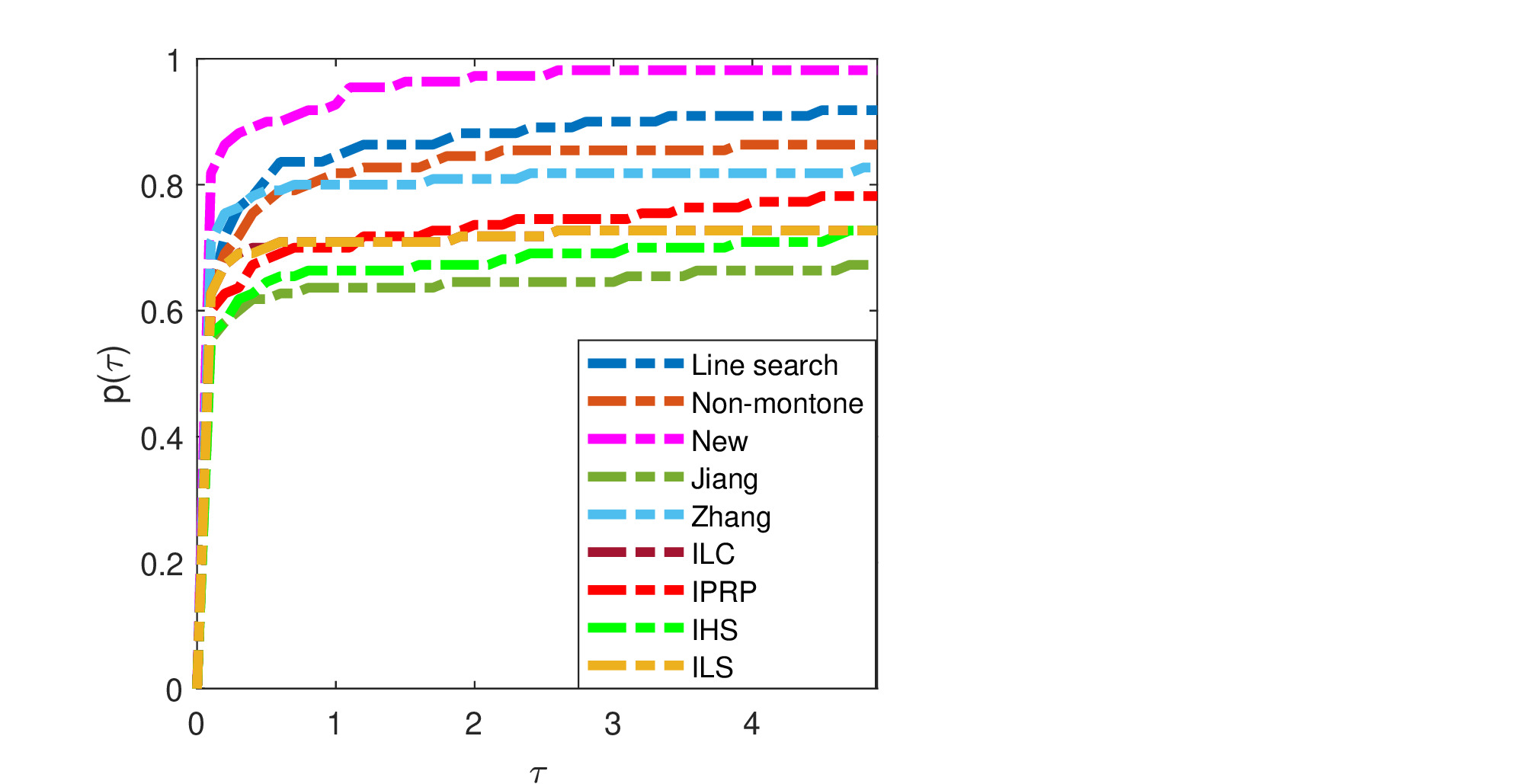}
\caption{(Left): Performance profiles of the total number of outer iterations, (Right): Performance profiles of CPU Time.}
\label{results}
\end{figure}

We also apply our algorithm  to solve the  Non-Negative Matrix Factorization (NMF)
which has several applications in image processing such as face detection problems. 
Given a non-negative matrix $V\in\Bbb{R}^{m\times n}$, a NMF
finds two non-negative matrices
$W\in\Bbb{R}^{m\times k}$ and $H\in\Bbb{R}^{k\times n}$ with
$k\ll\min(m,n)$ such that $X\approx WH$.  This problem can be formulated as
\begin{equation}\label{opti-n}
\min_{W,H\geq0} F(W,H)=\frac{1}{2}\|V-WH\|_{F}^2.
\end{equation}
Equation \eqref{opti-n} is a non-convex optimization problem. We compare our method and Zhang's algorithm \citep{Zhang} on  some random datasets and reported these results in Appendix \ref{AppB}.

\section{Conclusion} In this paper, we introduced a new non-monotone conjugate gradient  algorithm based on efficient Barzilai-Borwein step size. We introduced a new non-monotone parameter based on gradient behaviour and determined by  a trigonometric function. We use a convex combination of the determined method to compute the step size value in each iteration.   We prove that the   proposed algorithm has global convergence. We implemented and tested our algorithm on a set of standard test problems and the non-negative matrix factorization problems. The   proposed algorithm can solve $98\%$ of the test problems for a set of standard test instances. For the non-negative matrix factorization, the results indicate that our algorithm is more efficient compared to  Zhang' s method \citep{Zhang}.

\bibliographystyle{unsrtnat}
\bibliography{references}

\begin{thebibliography}{22}
\providecommand{\natexlab}[1]{#1}
\providecommand{\url}[1]{\texttt{#1}}
\expandafter\ifx\csname urlstyle\endcsname\relax
  \providecommand{\doi}[1]{doi: #1}\else
  \providecommand{\doi}{doi: \begingroup \urlstyle{rm}\Url}\fi

\bibitem[Barzilai and Borwein(1988)]{Barzilai}
Jonathan Barzilai and Jonathan~M Borwein.
\newblock Two-point step size gradient methods.
\newblock \emph{IMA journal of numerical analysis}, 8\penalty0 (1):\penalty0
  141--148, 1988.

\bibitem[Nocedal and Wright(2006)]{Nocedal}
Jorge Nocedal and Stephen Wright.
\newblock \emph{Numerical optimization}.
\newblock Springer Science \& Business Media, 2006.

\bibitem[Hestenes et~al.(1952)Hestenes, Stiefel, et~al.]{hestenes}
Magnus~Rudolph Hestenes, Eduard Stiefel, et~al.
\newblock \emph{Methods of conjugate gradients for solving linear systems},
  volume~49.
\newblock NBS Washington, DC, 1952.

\bibitem[Fletcher and Reeves(1964)]{Fletcher64}
Reeves Fletcher and Colin~M Reeves.
\newblock Function minimization by conjugate gradients.
\newblock \emph{The computer journal}, 7\penalty0 (2):\penalty0 149--154, 1964.

\bibitem[Fletcher(2013)]{Fletcher13}
Roger Fletcher.
\newblock \emph{Practical methods of optimization}.
\newblock John Wiley \& Sons, 2013.

\bibitem[Polak and Ribiere(1969)]{prp}
Elijah Polak and Gerard Ribiere.
\newblock Note sur la convergence de m{\'e}thodes de directions conjugu{\'e}es.
\newblock \emph{ESAIM: Mathematical Modelling and Numerical
  Analysis-Mod{\'e}lisation Math{\'e}matique et Analyse Num{\'e}rique},
  3\penalty0 (R1):\penalty0 35--43, 1969.

\bibitem[Grippo et~al.(1986)Grippo, Lampariello, and Lucidi]{Grippo86}
Luigi Grippo, Francesco Lampariello, and Stephano Lucidi.
\newblock A nonmonotone line search technique for newton’s method.
\newblock \emph{SIAM journal on Numerical Analysis}, 23\penalty0 (4):\penalty0
  707--716, 1986.

\bibitem[Zhang and Hager(2004)]{Zhang04}
Hongchao Zhang and William~W Hager.
\newblock A nonmonotone line search technique and its application to
  unconstrained optimization.
\newblock \emph{SIAM journal on Optimization}, 14\penalty0 (4):\penalty0
  1043--1056, 2004.

\bibitem[Ahookhosh et~al.(2012)Ahookhosh, Amini, and Peyghami]{Ahookhosh122}
Masoud Ahookhosh, Keyvan Amini, and Mohammad~Reza Peyghami.
\newblock A nonmonotone trust-region line search method for large-scale
  unconstrained optimization.
\newblock \emph{Applied Mathematical Modelling}, 36\penalty0 (1):\penalty0
  478--487, 2012.

\bibitem[Li and Wan(2019)]{Li19}
Ting Li and Zhong Wan.
\newblock New adaptive barzilai--borwein step size and its application in
  solving large-scale optimization problems.
\newblock \emph{The ANZIAM Journal}, 61\penalty0 (1):\penalty0 76--98, 2019.

\bibitem[Nazareth(1999)]{Nazareth01}
JL~Nazareth.
\newblock Conjugate-gradient methods, encyclopedia of optimization, c. floudas
  and p. pardalos, eds, 1999.

\bibitem[Esmaeili and Kimiaei(2014)]{Esmaeili}
Hamid Esmaeili and Morteza Kimiaei.
\newblock A new adaptive trust-region method for system of nonlinear equations.
\newblock \emph{Applied mathematical modelling}, 38\penalty0 (11-12):\penalty0
  3003--3015, 2014.

\bibitem[Ahookhosh and Amini(2012)]{Ahookhosh_Nu}
Masoud Ahookhosh and Keyvan Amini.
\newblock An efficient nonmonotone trust-region method for unconstrained
  optimization.
\newblock \emph{Numerical Algorithms}, 59\penalty0 (4):\penalty0 523--540,
  2012.

\bibitem[Amini and Ahookhosh(2014)]{Amini_App14}
Keyvan Amini and Masoud Ahookhosh.
\newblock A hybrid of adjustable trust-region and nonmonotone algorithms for
  unconstrained optimization.
\newblock \emph{Applied Mathematical Modelling}, 38\penalty0 (9-10):\penalty0
  2601--2612, 2014.

\bibitem[Ahookhosh and Ghaderi(2017)]{Ahookhosh15}
Masoud Ahookhosh and Susan Ghaderi.
\newblock On efficiency of nonmonotone armijo-type line searches.
\newblock \emph{Applied Mathematical Modelling}, 43:\penalty0 170--190, 2017.

\bibitem[Amini et~al.(2014)Amini, Ahookhosh, and Nosratipour]{Amini14}
Keyvan Amini, Masoud Ahookhosh, and Hadi Nosratipour.
\newblock An inexact line search approach using modified nonmonotone strategy
  for unconstrained optimization.
\newblock \emph{Numerical Algorithms}, 66\penalty0 (1):\penalty0 49--78, 2014.

\bibitem[Jiang and Jian(2019)]{Jiang}
Xianzhen Jiang and Jinbao Jian.
\newblock Improved fletcher--reeves and dai--yuan conjugate gradient methods
  with the strong wolfe line search.
\newblock \emph{Journal of Computational and Applied Mathematics},
  348:\penalty0 525--534, 2019.

\bibitem[Hager and Zhang(2005)]{Zhang}
William~W Hager and Hongchao Zhang.
\newblock A new conjugate gradient method with guaranteed descent and an
  efficient line search.
\newblock \emph{SIAM Journal on optimization}, 16\penalty0 (1):\penalty0
  170--192, 2005.

\bibitem[Dolan and Mor{\'e}(2002)]{Dolan}
Elizabeth~D Dolan and Jorge~J Mor{\'e}.
\newblock Benchmarking optimization software with performance profiles.
\newblock \emph{Mathematical programming}, 91\penalty0 (2):\penalty0 201--213,
  2002.

\bibitem[Andrei(2008)]{Andrei}
Neculai Andrei.
\newblock An unconstrained optimization test functions collection.
\newblock \emph{Adv. Model. Optim}, 10\penalty0 (1):\penalty0 147--161, 2008.

\bibitem[Han et~al.(2009)Han, Neumann, and Prasad]{Han}
Lixing Han, Michael Neumann, and Upendra Prasad.
\newblock Alternating projected barzilai-borwein methods for nonnegative matrix
  factorization.
\newblock \emph{Electron. Trans. Numer. Anal}, 36\penalty0 (6):\penalty0
  54--82, 2009.

\bibitem[Lee and Seung()]{Lee01}
DD~Lee and HS~Seung.
\newblock Algorithms for non-negative matrix factorization. nips (2000).
\newblock \emph{Google Scholar}, pages 556--562.

\end{thebibliography}
\section{Appendix A}\label{AppA}

\subsection{Algorithm}
In this section, we describe the new non-monotone conjugate gradient algorithm. Algorithm \ref{alg1} consists of two loops, inner and outer loop. In each inner loop, the value of step size by using non-monotone line search strategy is computed and then the new point, search direction, and conjugate gradient parameter are  calculated.

\begin{algorithm}[H]
\DontPrintSemicolon
  \KwInput{$x_0\in \Bbb{R}^n,~~\epsilon>0,~ \rho\in(0,1),~~
            \gamma\in (0,1),$ and $N=5$}
  \KwOutput{$x^*$}
  \Kwset{$k=0$}
 \While{$\|g_k\|\geq \epsilon$}{
 \While{Eq. (\ref{amin}) is False}{$\alpha\leftarrow\rho\alpha$}
 $\alpha_k\leftarrow\alpha$ \;
  Generate the new point $x_{k+1}=x_k+\alpha_k d_k$\;
  Compute the parameter $\beta_{k}$ by (\ref{cgpar})\;
   Generate $d_{k+1}$ by (\ref{dk})\;
   Calculate $\eta_{k+1}$ by using (\ref{eq:etakn})\;
   $k\leftarrow k+1$
 }
	\caption{\textbf{An improved non-monotone conjugate gradient method}}
		\label{alg1}
	\end{algorithm}

\subsection{Convergence} \label{s:convergence}
The proofs of the various lemmas and the main theorem are presented in this section.
\begin{proof} Proof of Lemma \ref{decent}:\\
If $k=0$, we have
	$$d_k^Tg_k=-\|g_k\|^2<0. $$
If $k\geq 1$, using (\ref{dk}) and (\ref{cgpar}), and we have:
\begin{eqnarray*}
d_k^Tg_k=-\|g_k\|^2+\omega\frac{\|g_k\|}{\|d_{k-1}\|}d_{k-1}^Tg_k.
\end{eqnarray*}	
Using the Cauchy-Schwarz inequality and we have:	
\begin{eqnarray}
	d_k^Tg_k\leq-\|g_k\|^2+\omega \|g_k\|^2=-(1-\omega)\|g_k\|^2.
\end{eqnarray}	    
\end{proof}
 \begin{proof} Proof of Lemma \ref{aboveserch}:\\
    The proof is obtained by combining (\ref{dk}) and the triangle inequality, that is:
    \begin{eqnarray*}
    \|d_k\|\leq \|g_k\|+\beta_k \|d_{k-1}\|=(1+\omega)\|g_k\|.
    \end{eqnarray*}
\end{proof}
To prove the convergence results, we need the following elementary lemmas. 
\begin{lemma}\label{de}
	Suppose that the sequence $\{x_k\}$ is generated by Algorithm 1. Then, $f_{l_k}$ is a decreasing sequence.
\end{lemma}
\begin{proof}
	We use the definition $R_k$ and (\ref{flk}), which imply that
	\begin{equation}\label{rk}
	R_k=\eta_kf_{l(k)}+(1-\eta_k)f_k\leq\eta_k f_{l(k)}+(1-\eta_k)f_{l(k)}=f_{l(k)}.
	\end{equation}
	It follows that:
	\begin{equation*}\label{dec}
	f_{k+1}\leq R_k+\gamma \alpha_k g_k^T d_k\leq f_{l(k)}
	+\gamma \alpha_k g_k^T d_k.
	\end{equation*}
	By using the Lemma \ref{decent}, we  can conclude that:
	\begin{equation}\label{dec-f}
	f_{k+1}\leq f_{l(k)}.
	\end{equation}
	On the other hand, from (\ref{flk}) we have:
	\begin{eqnarray*}
		f_{l(k+1)}&=&\max_{0\leq j\leq m(k+1)}\left\{f_{k+1-j}\right\}\\
		&\leq&\max_{0\leq j\leq m(k)+1}\left\{f_{k+1-j}\right\}=\max\left\{f_{l(k)},f_{k+1}\right\}.
	\end{eqnarray*}
	By applying (\ref{dec-f}), that is $f_{k+1}\leq f_{l(k)}$, we   conclude that $f_{l(k+1)}\leq f_{l(k)}$.
	This  shows that the sequence $f_{l(k)}$ is a decreasing sequence.
\end{proof}
\begin{lemma}
	Suppose that the sequence $\{x_k\}$ is generated by Algorithm 1. Then, for all $k\geq {0}$, we have $x_k\in \mathcal{L}(x_0)$.
\end{lemma}
\begin{proof}
	Definition of $f_{l(k+1)}$ implies that $f_{k+1}\leq f_{l(k+1)}$ for any $k\geq 0$.
	Therefore, we have:
	\begin{eqnarray}
	f_{k+1}&=&\eta_{k+1}f_{k+1}+(1-\eta_{k+1})f_{k+1}\nonumber\\
	&\leq& \eta_{k+1}f_{l({k+1})}+(1-\eta_{k+1})f_{k+1}=R_{k+1},\qquad \forall~k\in N_k\label{rkk}
	\end{eqnarray}
	By using definition of $R_k$, we can conclude that $R_0=f_0$. Now, by induction,
	assuming $x_i\in \mathcal{L}(x_0)$, for all $i=1,2,\ldots,k$, we show that $x_{k+1}\in \mathcal{L}(x_0)$. Relations
	(\ref{flk}) and (\ref{rk}) together with  Lemma \ref{de} imply that:
	
	\begin{equation*}
	f_{k+1}\leq f_{l(k+1)}\leq f_{l(k)}\leq f_0,
	\end{equation*}
	which implies that the sequence ${x_k}$ is contained in $\mathcal{L}(x_0)$.  
\end{proof}
The next part of  this section   describes some convergence results for the Algorithm \ref{alg1}.
\begin{lemma}
	Suppose that Algorithm \ref{alg1} generates the sequence $\{x_k\}$ and ($H1$)--($H2$) hold. Then, the sequence $\{f_{l(k)}\}$ is convergent.
\end{lemma}
\begin{proof} By using Lemma \ref{de} and the fact that $f_{l(0)}=f_0$ imply that the sequence $\{x_{l(k)}\}$
	remains in level set $\mathcal{L}(x_0)$. On the other hand, this fact $f(x_k)\leq f(x_{l(k)})$ proves that the sequence
	$\{x_{k}\}$ remains in $\mathcal{L}(x_0)$. Therefore, ($H1$) together with Lemma \ref{de} imply that the sequence
	$\{f_{l(k)}\}$ is convergent.
\end{proof}
\begin{lemma}\label{inf}
	Suppose that (H1) holds and the direction $d_k$ satisfies the first item of Lemma \ref{de}. Then for the
	sequence $\{x_{k}\}$  generated by Algorithm \ref{alg1}, we have:
	\begin{equation*}
	\lim_{k\rightarrow \infty}f_{l(k)}= \lim_{k\rightarrow \infty}f_k.
	\end{equation*}
\end{lemma}
\begin{proof}
	The proof is similar to Lemma 2 in \citep{Amini14}. Therefore, we omit it here.
\end{proof}

\begin{lemma}\label{rinf}
	Suppose that (H1) holds and the direction $d_k$ satisfies the first item of Lemma \ref{de}. If the
	sequence $\{x_{k}\}$  generated by Algorithm \ref{alg1}, then we have:
	\begin{equation*}
	\lim_{k\rightarrow \infty}R_k= \lim_{k\rightarrow \infty}f_k.
	\end{equation*}
\end{lemma}
\begin{proof}
	Using (\ref{rk}) and (\ref{rkk}), we conclude that:
	\begin{equation*}
	f_k\leq R_k\leq f_{l(k)}
	\end{equation*}
	By applying  Lemma \ref{inf}, we obtain the result.
\end{proof}
Now, we can prove the Lemma \ref{low-bou}. 
\begin{proof} Proof of Lemma \ref{low-bou}:\\
	We consider two cases:
	If $\frac{\alpha_k}{\rho}\geq\beta_1$, which implies that $\alpha_k\geq \beta_1\rho$ and it completes the proof.
	Now, we assume that $\frac{\alpha_k}{\rho}<\beta_1$. In this case we have  $\alpha_k< \beta_1\rho$. Therefore, the non-monotone
	condition does not hold, i.e.,
	\begin{equation}\label{nonmonotone}
	f(x_k+\frac{\alpha_k}{\rho} d_k)> R_k+\gamma \frac{\alpha_k}{\rho} g_k^T d_k
	\end{equation}
	Now, by using the mean value theorem, i.e., Cauchy–Schwarz inequality, we conclude that:
	\begin{eqnarray}
	f(x_k+\alpha d_k)&=&f(x_k)+\alpha g_k^Td_k+
	\int_0^{1}\alpha(g(x_k+t\alpha d_k)-g_k)^Td_k dt\nonumber\\
	&\leq& f(x_k)+\alpha g_k^Td_k+\alpha\|d_k\|\int_0^{1}\|(g(x_k+t\alpha d_k)-g_k^T)\|dt\nonumber\\
	\end{eqnarray}
	Now, by using ($H2$) and the fact that $R_k\geq f_k$, we have:
	\begin{eqnarray}\label{cauchy}
	f(x_k+\alpha d_k)&\leq& f(x_k)+\alpha g_k^Td_k+L\alpha^2\|d_k\|^2\int_0^{1}t dt\nonumber\\
	& =& f(x_k)+\alpha g_k^Td_k+\frac{L}{2}\alpha^2\|d_k\|^2\nonumber\\
	&\leq& R_k+\alpha g_k^Td_k+\frac{L}{2}\alpha^2\|d_k\|^2
	\end{eqnarray}
	By putting $\alpha=\frac{\alpha_k}{\rho}$ and combining  (\ref{nonmonotone}) with (\ref{cauchy}),  we conclude that:
	\begin{eqnarray*}
		R_k+\frac{\alpha_k}{\rho} g_k^Td_k+\frac{L}{2\rho^2}\alpha_2^2\|d_k\|^2\geq R_k+\gamma \frac{\alpha_k}{\rho} g_k^T d_k.
	\end{eqnarray*}
    Using Lemmas \ref{decent} and \ref{aboveserch}, we conclude that:
	\begin{equation}\label{lowlam}
	\frac{L}{2\rho}(1+\omega)^2\alpha_2^2\|g_k\|^2\geq\frac{L}{2\rho}\alpha_2^2\|d_k\|^2>-(1-\gamma)g^T_kd_k\geq (1-\omega)(1-\gamma)\|g_k\|^2.
	\end{equation}
	It implies that:
	\begin{equation*}
	\alpha_k\geq \frac{2(1-\omega)\rho(1-\gamma)}{L(1+\omega)^2}.
	\end{equation*}
 
\end{proof}

Now, we prove the Theorem \ref{glob}, which shows that Algorithm \ref{alg1} has global convergence. 
\begin{proof} Proof of Theorem \ref{glob}:\\
	We have:
	\begin{eqnarray}\label{con1}
	R_k+\gamma\alpha_kg^T_k d_k &\geq& f_{k+1}\nonumber\\
	\Rightarrow R_k-f_{k+1}&\geq&-\gamma\alpha_kg^T_k d_k\geq c_1\gamma\alpha_k\|g_k\|^2\geq0.
	\end{eqnarray}
	This fact together with Lemma \ref{rinf} imply
	that $ \lim_{k\rightarrow \infty} \|g_k\|=0$.
	 This shows that proposed 
	 algorithm has global convergence.
\end{proof}

\section{Appendix B}\label{AppB}
\subsection{Numerical Results}
Here, we present some implementation level details. Since our algorithm improves the non-monotone scheme,  we chose two algorithms from the non-monotone line search category and six algorithms from the  Wolfe line search area for performing the comparison. We  select the following two state of art algorithms in the non-monotone category.
These algorithms calculate the value of step size using a the non-monotone line search strategy in each iteration.
\begin{description}
	\item[$\bullet$] Ahookhosh's strategy in \citep{Ahookhosh122}
	\item[$\bullet$] Amini's strategy in \citep{Amini14}
\end{description}
Following six other algorithms that use the Wolfe line search conditions to compute step size in each iteration are used for comparison.
\begin{description}
	\item[$\bullet$] $\beta^{Jiang}$ proposed in \citep{Jiang}
	\item[$\bullet$] $\beta^{Zhang}$ proposed in \citep{Zhang}
	\item[$\bullet$] $\beta^{ILC}$ proposed in \citep{Jiang}
	\item[$\bullet$] $\beta^{IPRP}$ proposed in \citep{Jiang}
	\item[$\bullet$] $\beta^{IHS}$ proposed in \citep{Jiang}
	\item[$\bullet$] $\beta^{ILS}$ proposed in \citep{Jiang}
\end{description}
All algorithms were coded in MATLAB 2017 environment and tested on a laptop (Intel(R) Core(TM) i5-7200U CPU 3.18 GHz with 12GB  RAM). 
For all algorithms, we use the following initial values.
\begin{eqnarray*}
	\gamma=10^{-4},\qquad N=5,\qquad \rho=0.75,\qquad c=10^{-4}.
\end{eqnarray*}
All the experiments terminate when the following conditions are met:
\begin{description}
	\item[$\bullet$] $\|g_k\|<10^{-6}$
	\item[$\bullet$] The number of iterations is greater than 20000.
\end{description}
As   parameter $\omega$ can take positive real values in the interval $(0,1)$, there are many choices for it. We tried several strategies and chosen the following strategy, which perform better. The key idea of this choice is   from the structure of the conjugate gradient parameter proposed by Jiang \citep{Jiang}.
\begin{eqnarray*}
\omega_k=\left\{ 
\begin{array}{lr}
	0.001\qquad\qquad\quad\qquad if~\frac{| g^T_kd_{k-1}|}{-g^T_{k-1}d_{k-1}}\leq0,&\\
		0.999\qquad\qquad\quad\qquad if~\frac{| g^T_kd_{k-1}|}{-g^T_{k-1}d_{k-1}}\geq 1,&\\
	\frac{| g^T_kd_{k-1}|}{-g^T_{k-1}d_{k-1}}\qquad\qquad\qquad otherwise.&\\
\end{array} \right.
\end{eqnarray*}

The following subsection presents the results on a set of $110$ standard test problems.
\subsection{A set of standard test problems}
We run all the algorithms on 110 standard test problems from \citep{Andrei} with dimensions ranging between $2$ to $5,000,000$. When the algorithm stops under the second condition, i.e., the number of iterations  is greater than $20000$,  the method is deemed to fail for solving the corresponding test problem.
The comparison between considered algorithms is based on the number of function evaluations, the number of gradient evaluations,
the number of iterations, and the CPU time(s). 
\\
To visualize the complete behaviour of the algorithms, we use the  performance profiles proposed by  Dolan and Moré \citep{Dolan}.
Note that the performance profile $p_s(\tau): \Bbb{R}\mapsto [0, 1]$ for an algorithm is a non-decreasing, piece-wise constant function, continuous from the right at each breakpoint. Moreover, the value of $p_s(1)$ denotes the probability that the algorithm will win over the rest.

Suppose that $K$ is a set of $n_k$ test functions and $S$ is a set of  $n_s$  solvers. For $s\in S$ and function $k\in K$, consider $a_{p,s}$ as  the number of gradient evaluations, objective function evaluations, CPU Time, or the number of iterations required to solve function $k \in K$ by algorithm  $s\in S$. Then the algorithms comparison
is based on the performance ratio as follows:
\begin{equation*}
r_{k,s}=\frac{a_{k,s}}{\min\{a_{k,s},~k\in K,~s\in S\}}
\end{equation*}
We obtain the overall evaluation of each algorithm by:
\begin{eqnarray}
p_s(\tau)= \frac{1}{n_k}~size~\{k\in K: r_{k,s}\leq \tau\}
\end{eqnarray}
In general, solvers with high values of $p_s(\tau)$ or in the upper right of the figure represent the best algorithm.

The performance profile  in terms of   function evaluations and the number of gradient evaluations are presented in figures \ref{fun} and \ref{gri}, respectively.

\begin{figure}[h!]
	\centering
	\includegraphics[width=.95\textwidth]{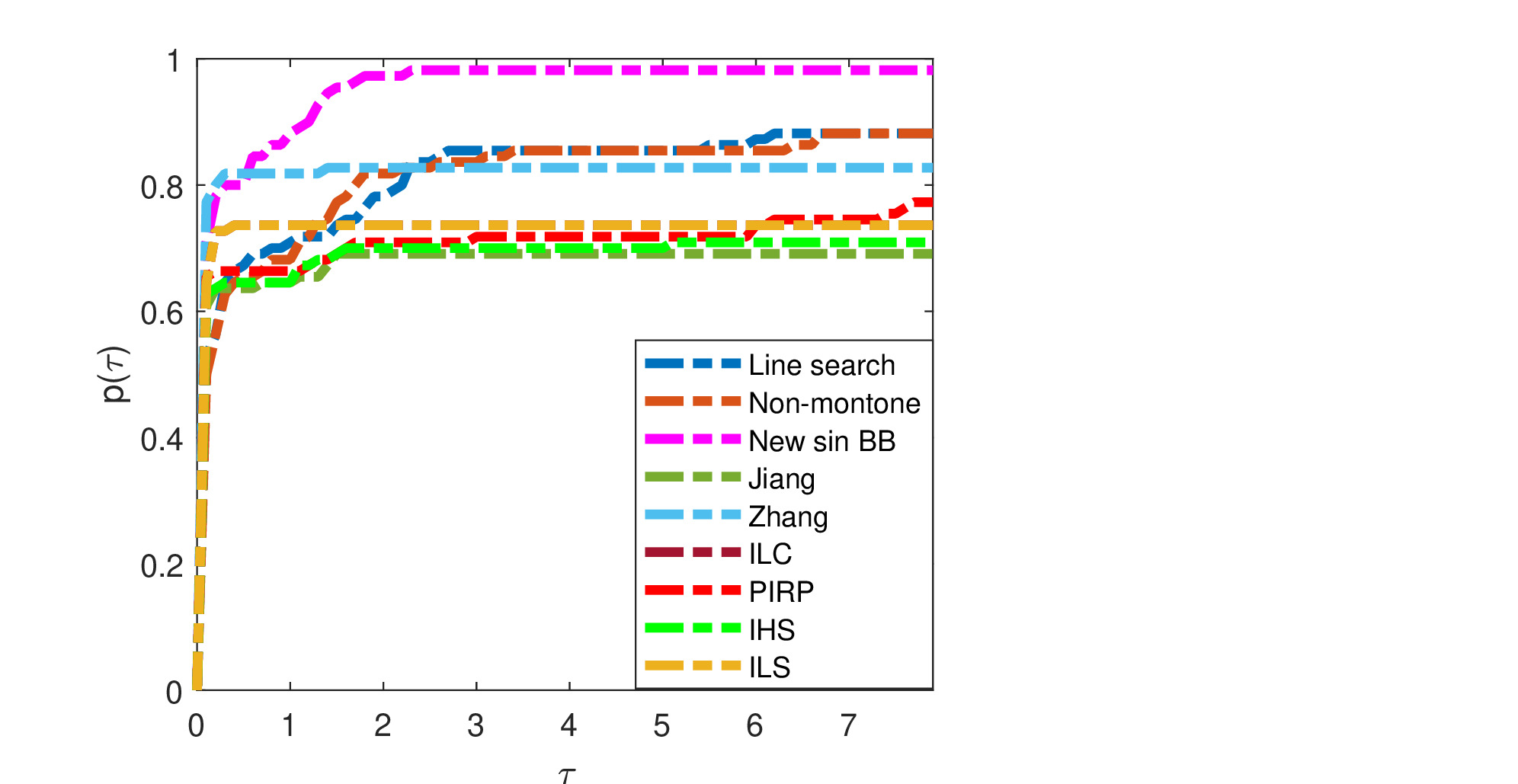}
	\caption{Performance profiles of the total number of function evaluations}
	\label{fun}
\end{figure}

\begin{figure}[h!]
	\centering
	\includegraphics[width=.95\textwidth]{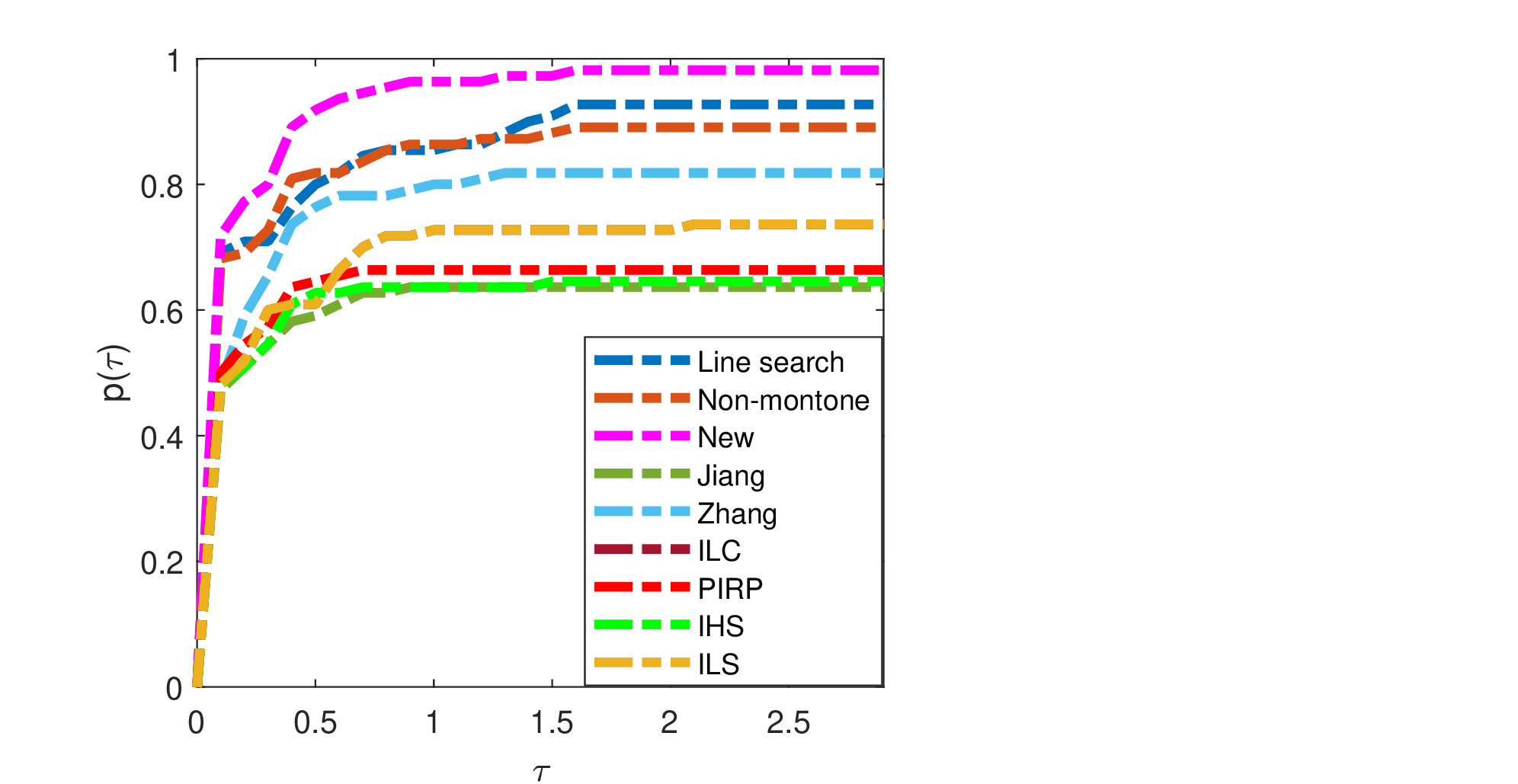}
	\caption{Performance profiles of the total number of  gradient evaluations}
	\label{gri}
\end{figure}

Figures \ref{out} and \ref{cpu} show the performance profile  in terms of the number of iterations and the CPU time(s) for the proposed algorithm and eight other algorithms.

\begin{figure}[h!]
	\centering
	\includegraphics[width=.95\textwidth]{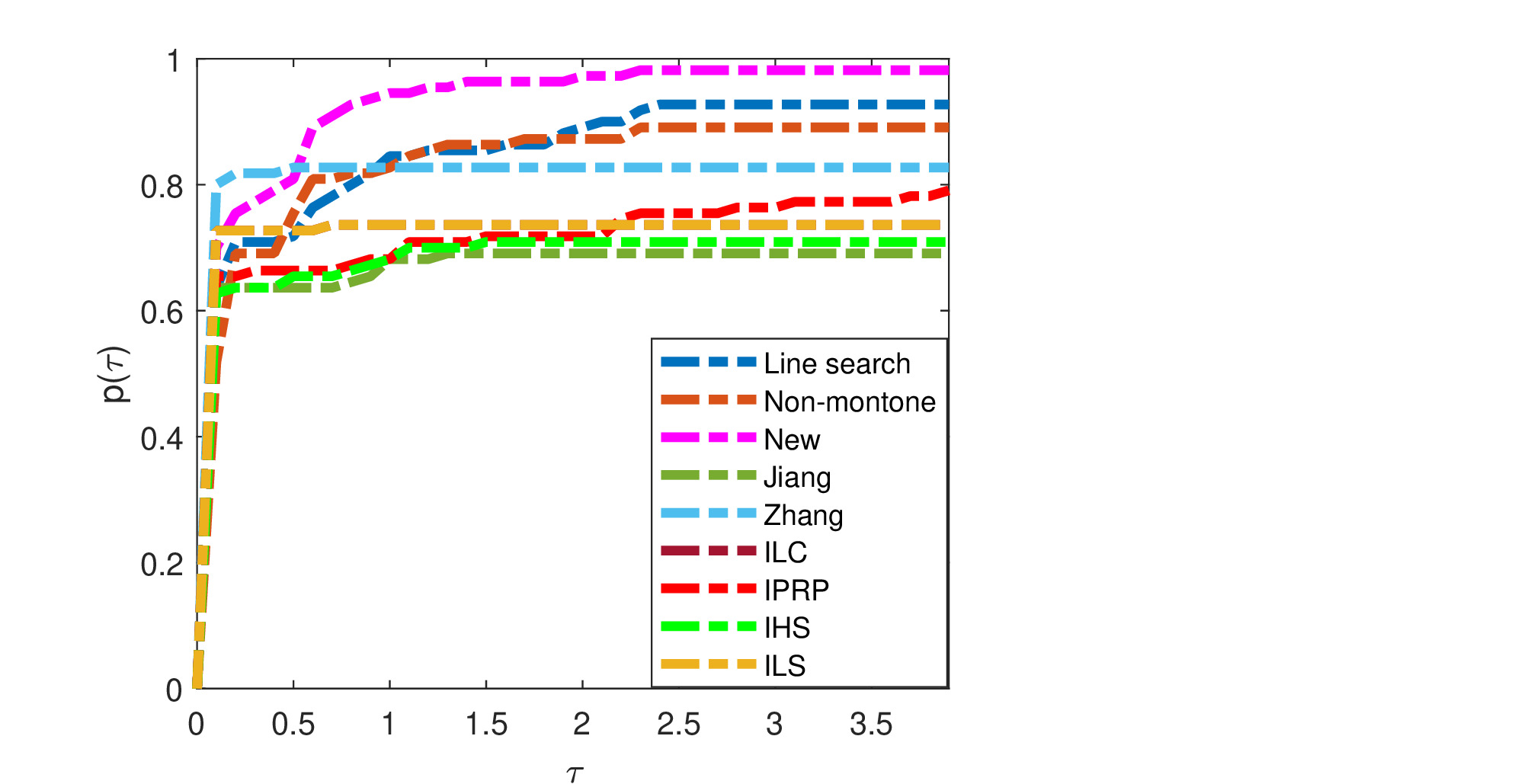}
	\caption{Performance profiles of the total number of outer iterations}
	\label{out}
\end{figure}


\begin{figure}[h!]
	\centering
	\includegraphics[width=.95\textwidth]{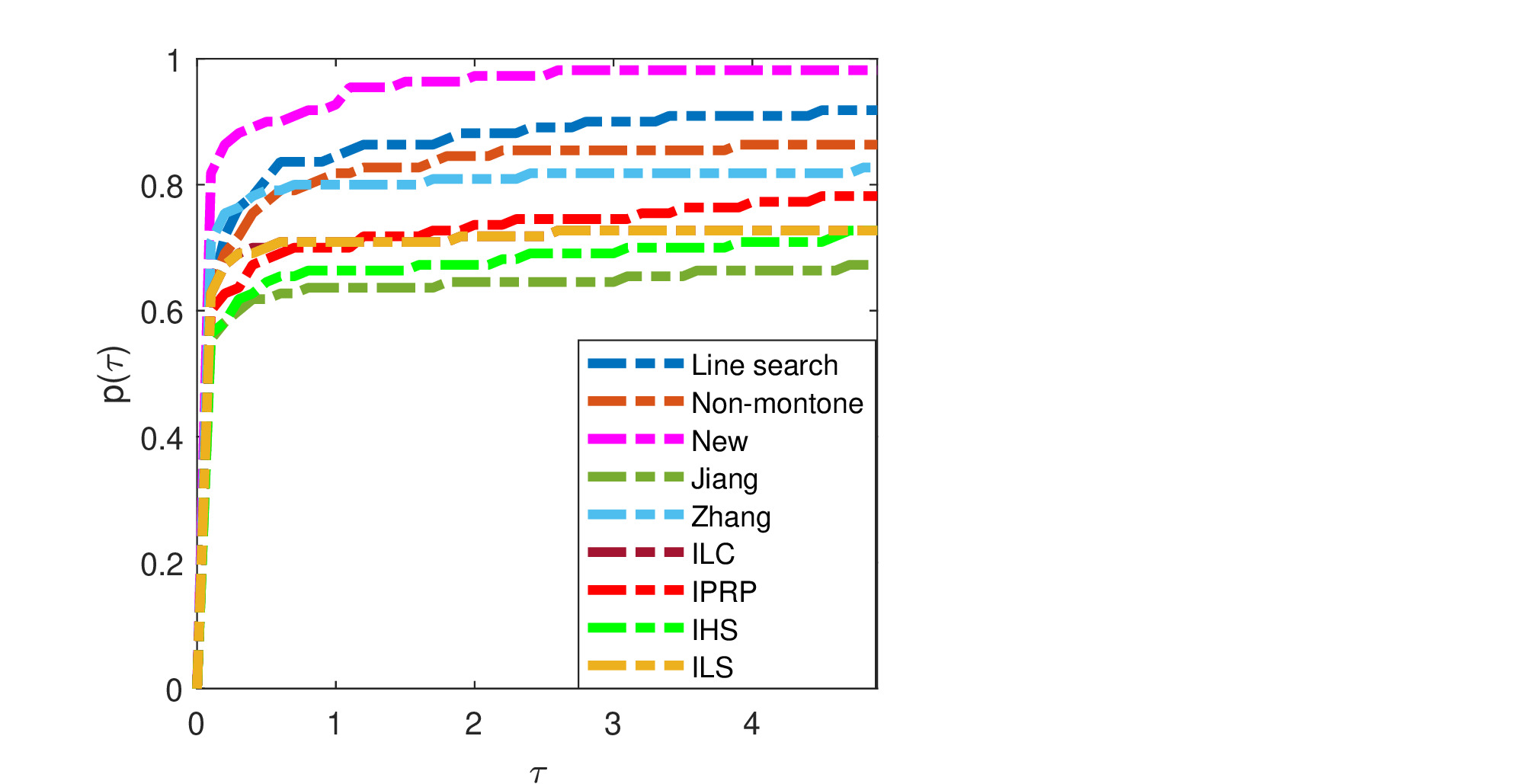}
	\caption{Performance profiles of CPU Time}
	\label{cpu}
\end{figure}

We conclude from the figures that the  proposed algorithm can solve $98\%$ of the test problems. The performance profiles for the number of iterations,   total CPU, time,   number of gradient evaluations, and   number of function evaluations indicate that  the proposed method has a high  computational performance compared to the other methods.

\subsection{Some non-negative matrix factorization test problems}
Here, we apply our algorithm to Non-Negative Matrix Factorization (NMF)
which has several applications in image processing, such as face detection problems. 
 Given a non-negative matrix $V\in\Bbb{R}^{m\times n}$, a NMF
finds two non-negative matrices
$W\in\Bbb{R}^{m\times k}$ and $H\in\Bbb{R}^{k\times n}$ with
$k\ll\min(m,n)$ such that
\begin{equation}\label{non-neg}
V\approx WH.
\end{equation}
This problem can be formulated as a non-convex optimization problem:
\begin{equation}\label{opti-nmf}
\min_{W,H\geq0} F(W,H)=\frac{1}{2}\|V-WH\|_{F}^2.
\end{equation}
 In recent years, several iterative approaches have been introduced for solving (\ref{opti-nmf}), for example, \citep{Han,Lee01}. The alternating non-negative least squares (ANLS)
framework is a popular approach for solving  (\ref{opti-nmf}), which finds the optimal solution by solving the following two convex sub-problems:
\begin{eqnarray}\label{nmf-1}
W^{k+1} =\text{arg}\min_{W\geq0}F(W,H^k) =\frac{1}{2}\|V-WH^k\|_{F}^2,
\end{eqnarray}
and
\begin{eqnarray}\label{nmf-2}
H^{k+1} =\text{arg}\min_{H\geq0}F(W^{k+1},H) =\frac{1}{2}\|V-W^{k+1}H\|_{F}^2.
\end{eqnarray}
To solve this problem, we use the following strategy:
\begin{description}
	\item[S0]Algorithm starts with the initial point, i.e., $\bar{W}\geq0$ and $\bar{H}\geq0$, set $k = 0$.
	\item[S1] Stop if $
	\|[\nabla_HF(\bar{W}^k,\bar{H}^k),\nabla_WF(\bar{W^k},\bar{H}^k)]\|_F\leq\epsilon
	\|[\nabla_HF(\bar{W^0},\bar{H^0}),\nabla_WF(\bar{W^0},\bar{H}^0)]\|_F.
$
	\item[S2] To get $W^{k+1}$, solve the sub-problem: $\min_{W\geq0}F(W,\bar{H}^k) =\frac{1}{2}\|V-W\bar{H}^k\|_{F}^2$.
	\item[S4] Set $\bar{W}^{k+1}=W^{k+1}$.
	\item[S2] To get $H^{k+1}$, solve the sub-problem: $\min_{H\geq0}F(\bar{W}^{k+1}, H) =\frac{1}{2}\|V-\bar{W}^{k+1}H\|_{F}^2$. 
	\item[S5] Set $\bar{H}^{k+1}=H^{k+1}$.
	\item[S6] Set $k:=k+1$ and go to S1.
\end{description}
Now, we use the above setup to solve some NMF problems using our algorithm and compare it to Zhang's algorithm \citep{Zhang} which had the best results  for solving a set of standard test problems.  To this end, we generate a random matrix $V $ as random with elements in $[0,1]$. We run the algorithm for  matrices with ranks  $\{5,10,15,20,40,50\}$. For each case, we run each of the algorithms 10 times. We, calculated  the average of the results and presented them in Table \ref{tab-rand}. In this table, $m$ and  $n$ denote the number of rows and columns of matrix  $V$, $k$ 
 denote the matrix rank. The number of outer iterations  is denoted by ``Iter''. We use the ``Niter'' for   the number of inner iterations. The value of gradient is denoted by ``Pgn''. The CPU time and error for each of the algorithms are denoted by ``Time'' and ``Error'' respectively.  

%
\begin{table}[]
    \centering
    \begin{tabular}[hp]{abababa}
\hline
{$m$ \,\,\,\,\,\, $n$ \,\,\,\,\,\, $k$} & {${Iter}$} & {${Niter}$} & {${Pgn}$}
& {${Time}$} & {$Error$} & {$Algorithm$} \\ \hline
{$50\times25\times5$}        &29.20 & 187.30 & 0.0045 & 0.02 & 0.014 & {Zhang}\\
                              & 26.70 & 72.60 & 0.0043 & 0.02 & 0.014 & {New}  \\
                          \hline
{$100\times50\times5$}          & 24.20 & 165.30 & 0.058 & 0.03 & 0.12 & {Zhang} \\
                                  & 22.30 & 85.80 & 0.033 & 0.01 & 0.12 & {New} \\
                                  \hline
{$100\times 200\times 15$}        & 25.30 & 196.10 & 0.015 & 0.137 & 0.086 & {Zhang} \\
                                  & 17.50 & 112.80 & 0.018 & 0.032 & 0.085 & {New} \\
                                  \hline
  {$200\times100\times10$}        & 19.00 & 134.80 & .093 & 0.10 & 0.08 & {Zhang} \\
                                  & 17.10 & 103.40 & .073 & 0.03 & 0.08 & {New} \\
                                 \hline
  {$300\times100\times20$}        & 25.00 & 212.40 & 0.55 & 0.32 & 0.08 & {Zhang} \\
                                  & 20.60 & 131.50 & 0.44 & 0.15 & 0.08 & {New}\\
                                   \hline
  {$300\times500\times20$}         &29.00 & 128.90 & 1.98 & 0.75 & 0.054& {Zhang} \\
                                  & 25.80 & 90.80 & 0.96 & 0.14 & 0.054 &{New} \\
                                   \hline
  {$500\times100\times20$}        & 36.10 & 231.80 & 7.6 & 0.48 & 0.06& {Zhang} \\
                                  & 31.70 & 90.70 & 4.8 & 0.12 & 0.05 & {New} \\
                                   \hline
  {$1000\times500\times50$}       &36.10  & 187.50  & 18.5 & 2.89 & 0.04 & {Zhang}\\
                                  & 32.40 & 130.80  & 12.75 & 1.06 & 0.03 & {New}\\
                                   \hline

    \end{tabular}
    \caption{The results of performing  new algorithm and Zhang's algorithm on  some  random datasets}
\label{tab-rand}
\end{table}
As we see that in  most cases, the proposed algorithm performs better than previous best algorithm due to Zhang \citep{Zhang}.



\end{document}